\documentclass[12pt]{article}
\font\smallit=cmti10

\usepackage{graphicx, amsmath, amssymb, amsthm, tikz, hyperref}
\usepackage{enumitem}

\usepackage{mathtools}

\DeclarePairedDelimiter{\floor}{\lfloor}{\rfloor}

\usepackage{epstopdf}
\usepackage{xcolor}

\makeatletter

\renewcommand\section{\@startsection {section}{1}{\z@}
{-30pt \@plus -1ex \@minus -.2ex}
{2.3ex \@plus.2ex}
{\normalfont\normalsize\bfseries\boldmath}}

\renewcommand\subsection{\@startsection{subsection}{2}{\z@}
{-3.25ex\@plus -1ex \@minus -.2ex}
{1.5ex \@plus .2ex}
{\normalfont\normalsize\bfseries\boldmath}}

\renewcommand{\@seccntformat}[1]{\csname the#1\endcsname. }

\makeatother
\newtheorem{thm}{Theorem}[section]

\newtheorem{lemm}[thm]{Lemma}
\newtheorem{cor}[thm]{Corollary}

\newtheorem{definition}[thm]{Definition}
\title{\textsc{flipping coins}}

\newcommand{\GL}{{G^\mathcal{L}}}
\newcommand{\GR}{{G^\mathcal{R}}}
\newcommand{\HL}{H^\mathcal{L}}
\newcommand{\HR}{H^\mathcal{R}}
\newcommand{\pos}[1]{\underline{#1}}
\newcommand{\name}{\textsc{flipping coins}}
\begin{document}
\begin{center}
\uppercase{The game of Flipping coins}
\vskip 20pt
{\bf Anthony Bonato\footnote{Supported by an NSERC Discovery grant.}}\\
{\smallit Ryerson University, Toronto, Ontario, Canada}\\
{\tt abonato@ryerson.ca}\\
\vskip 10pt
{\bf  Melissa A. Huggan\footnote{Supported by an NSERC Postdoctoral fellowship.}}\\
{\smallit Ryerson University, Toronto, Ontario, Canada}\\
{\tt melissa.huggan@ryerson.ca}\\
\vskip 10pt
{\bf Richard J. Nowakowski\footnote{Supported by an NSERC Discovery grant.}}\\
{\smallit Dalhousie University, Halifax, Nova Scotia, Canada}\\
{\tt r.nowakowski@dal.ca}\\
\end{center}
\vskip 20pt

\date{}

\begin{abstract} We consider \name, a partizan version of the impartial game \textsc{turning turtles}, played on lines of coins. We show the values of this game are numbers, and these are found by first applying a reduction, then decomposing the position into an iterated ordinal sum. This is unusual since moves in the middle of the line do not eliminate the rest of the line. Moreover, when $G$ is decomposed into lines $H$ and $K$, then $G=(H:K^R)$. This is in contrast to \textsc{hackenbush strings}
where $G= (H:K)$.
\end{abstract}


\section{Introduction}\label{secintro}
In Winning Ways Volume 3~\cite{Berlekamp}, Berlekamp, Conway, and Guy introduced \textsc{turning turtles} and considered many variants. Each game involves a finite row of turtles, either on feet or backs, and a move is to turn one turtle over onto its back, with the option of flipping a number of other turtles, to the left, each to the opposite of its current state (feet or back). The number depends on the rules of the specific game. The authors moved to playing with coins as playing with turtles is cruel.

These games can be solved using the Sprague-Grundy theory for impartial games \cite{Berlekamp}, but the structure and strategies of some variants are interesting. The strategy for: \textsc{moebius} (flip up to $5$ coins) played with 18 coins, involves \textit{M\"obius transformations};
\textsc{mogul} (flip up to $7$ coins) on $24$ coins, involves the \textit{miracle octad generator} developed by R. Curtis in his work on the Mathieu group $M_{24}$ and the Leech lattice, \cite{M24,Curtis}; \textsc{ternups} (flip three equally spaced coins) requires ternary expansions; and \textsc{turning corners}, a two-dimensional version where the corners of a rectangle are flipped,  needs nim-multiplication.

We consider a simple partizan version of \textsc{turning turtles}, also played with coins. We give a complete solution and show that it involves ordinal sums. This is somewhat surprising since moves in the middle of the line do not eliminate moves at the end. Compare this with \textsc{hackenbush strings} \cite{Berlekamp}, and
\textsc{domino shave} \cite{CarvalhoHNS1}.

We will denote heads by $0$ and tails by $1$.  Our partizan version will be played with a line of coins, represented by a $0$-$1$ sequence, $d_1d_2\ldots d_n$, where $d_i\in\{0,1\}$.
 To this position, we associate the binary number $\sum_{i=1}^{n}d_i2^{i-1}$. Left moves by choosing some pair of coins  $d_i,d_j$, $i<j$, where $d_i=d_j=1$ and flips them over so that both coins are $0$s. Right also chooses a pair $d_k,d_\ell$, $k<\ell$, with $d_k=0$ and $d_\ell=1$,
and flips them over. If $j$ is the greatest index such that $d_j=1$, then $d_k$, $k>j$, will be deleted.
For example,
\[1011 = \{0001, 001, 1 \mid 1101, 111\}.\]
The game eventually ends since the associated binary number decreases with every move.  We call this game \name.

 The game is biased to Left. If there are a non-zero even number of $1$s in a position, then Left always has a move; that is, she will win.
Left also wins any non-trivial position starting with $1$. However, there are positions that Right wins.
 The two-part method to find the outcomes and values of the remaining positions can be applied to all positions. First, apply a modification to the position (unless it is all 1s), which reduces the number of consecutive $1$s to at most three. After this reduction, build an iterated ordinal sum, by successively deleting everything after the third last $1$, this deleted position determines the value of the next term in the ordinal sum. As a consequence, the original position is a Right win, if the position remaining at the end is of the form $0\ldots01$, and the value is given by the ordinal sum.

Necessary background for numbers is in Section~\ref{sec:numbers}. Section~\ref{sec: main results} contains results about outcomes, and it also includes our main results. First, we show that the values are numbers in Theorem~\ref{thm:numbers2}. Next, an algorithm to find the value of a position is presented, and Theorem~\ref{thm: main theorem} states that the value given by the algorithm is correct.

The actual analysis is in Section~\ref{sec:proofs}. It starts by identifying the best moves for both players in Theorem~\ref{thm:bestmoves}. This leads directly to the core result Lemma~\ref{lem:osvalue}, which shows that the value of a position is an ordinal sum.  The ordinal sum decomposition of $G$ is found as follows.
Let $G^L$ be the position after the Left move that removes the rightmost $1$s. Let $H$ be the string $G\setminus G^L$; that is, the substring eliminated by Left's move. Let $H^R$ be the result of Right's best move in $H$. Now, we have that $G=G^L:H^R$. In contrast, the ordinal sums for \textsc{hackenbush strings} and \textsc{domino shave} \cite{CarvalhoHNS1},  involve the value of $H$ not $H^R$.

The proof of Theorem~\ref{thm: main theorem} is given in Section~\ref{sec: theorem proof}. The final section includes a brief discussion of open problems.\\

Finally we pose a question for the reader, which we answer at the end of Section~\ref{sec: theorem proof}: Who wins $0101011111+1101100111+0110110110111$ and how?

 \section{Numbers}\label{sec:numbers}
All the values in this paper are numbers and this section contains all the necessary background to make the paper self-contained. For further details, consult \cite{Albert,Siegel}. Positions are written in terms of their options, that is $G=\{\GL\mid \GR\}$.

\begin{definition}{\rm\cite{Albert,Berlekamp,Siegel}}
Let $G$ be a number whose options are numbers and let  $G^L$, $G^R$ be the Left and Right options of the canonical form of $G$.
\begin{enumerate}
\item If there is an integer $k$, $G^L<k<G^R$, or if either $G^L$ or $G^R$ does not exist, then $G$ is the integer, say $n$, closest to zero that satisfies $G^L<n<G^R$.
\item If both $G^L$ and $G^R$ exist and the previous case does not apply, then $G=\frac{p}{2^q}$, where $q$ is the least non-negative integer such that there is an odd integer $p$ satisfying $G^L <\frac{p}{2^q}<G^R$.
\end{enumerate}
\end{definition}

The properties of numbers required for this paper are contained in the next two theorems.

\begin{thm}\label{thm:literature}{\rm\cite{Albert,Berlekamp,Siegel}}
Let $G$ be a number whose options are numbers and let  $G^L$, $G^R$ be the Left and Right options of the canonical form of $G$. If $G'$ and $G''$ are any Left and Right options respectively, then \[G'\leqslant G^L<G<G^R\leqslant G''.\]
\end{thm}

Theorem~\ref{thm:literature} shows that if we know that the property holds, we need to only consider the best move for both players in a number.

We include the following examples.
 \begin{enumerate}[label=(\alph*)]
 \item $0=\{\,\mid \, \} = \{-9\mid \,\} = \{-\frac{1}{2}\mid \frac{7}{4}\}$;
 \item $-2 = \{\, \mid -1 \} = \{-\frac{5}{2} \mid  -\frac{31}{16}\}$;
 \item  $1= \{0\mid\, \} = \{0 \mid 100 \}$;
 \item  $\frac{1}{2}= \{0 \mid 1 \} = \{ \frac{3}{8} \mid \frac{17}{32}\}$.
 \end{enumerate}

For games $G$ and $H$, to show that $G \geq H$, we need to show that $G-H \geq 0$. Meaning, we need to show that $G-H$ is a Left win moving second. For more information, see Sections 5.1, 5.8, and 6.3 of the second edition of~\cite{Albert}.

 Let $G$ and $H$ be games. The \textit{ordinal sum} of $G$, the \textit{base}, and $H$, the \textit{exponent}, is
$$G:H = \{\GL, G:\HL\mid \GR, G:\HR\}.$$
Intuitively, playing in $G$ eliminates $H$ but playing in $H$ does not affect $G$. For ease of reading, if an ordinal sum is a term in an expression, then we enclose it in brackets.

Note that $x:0=x=0:x$ since neither player has a move in $0$. We demonstrate how to calculate the values of other positions with the following examples.

\begin{enumerate}[label=(\alph*)]
\item $1:1 = \{1\mid \,\}  = 2$;
\item $1:-1 = \{0\mid 1\} = \frac{1}{2}$;
\item $1:\frac{1}{2} = \{0, (1:0) \mid (1:1)\} = \{0, 1\mid \{1 \mid \,\}\} = \{ 1 \mid 2\} = \frac{3}{2}$;
\item $\frac{1}{2}:1 = \{0, (\frac{1}{2}:0)\mid 1\} = \{0,\frac{1}{2}\mid 1\} = \{\frac{1}{2}\mid 1\} = \frac{3}{4}$;
\item $(1:-1):\frac{1}{2} = (\frac{1}{2}:\frac{1}{2}) = \{0, (\frac{1}{2}:0) \mid 1, (\frac{1}{2}:1)\} = \{0, \frac{1}{2} \mid 1, \frac{3}{4} \} = \{\frac{1}{2}\mid \frac{3}{4}\} = \frac{5}{8}$.
\end{enumerate}

Note that in all cases, players prefer to play in the exponent rather than the base. This is true in all cases, but in this paper all the exponents will be positive.

\begin{thm}\label{thm:oso}
Let $G$ be a number all of whose options are numbers, and let $H\geqslant 0$ be a number.
\begin{enumerate}
\item $G^L<(G:H^L)<(G:H)<(G:H^R)<G^R$.
\item If $H=0$, then $G:H=G$. If $H> 0$, then $(G:H)>G$.
\end{enumerate}
\end{thm}

\begin{proof}
For item (1), by definition,  $G^L$ is a Left option of $G:H^L$ and both are Left options of $G$. Thus $G^L<(G:H^L)<G$ by Theorem~\ref{thm:literature}. The proof is similar for the Right options.

For item (2), if $H=0$, then $G:0=\{\GL\mid\GR\}$ and this is $G$. Suppose $H>0$. Let Right move first in $(G:H)-G$. If Right moves to $G^R-G$ or $(G:H)-G^L$, then Left responds to $G^R-G^R$ and $G^L-G^L$ respectively and wins. If Right moves to $(G:H^R)-G$, then since $H^R>0$,  $(G:H^R)-G>0$ by induction on the options.
\end{proof}

To prove that  all the positions are numbers, we use results from \cite{CarvalhoHNS}. A set of positions from a ruleset is called a \emph{hereditary closed set of positions of a ruleset} 
if it is closed under taking options. This game satisfies ruleset properties introduced in~\cite{CarvalhoHNS}. In particular, the properties are called the \emph{F1 property} and the \emph{F2 property} and are defined formally as follows.
\begin{definition}{\rm \cite{CarvalhoHNS}}
Let $S$ be a hereditary closed ruleset. Given a position $G \in S$, the pair $(G^{L}, G^{R}) \in G^{\mathcal{L}}\times G^{\mathcal{R}}$ satisfies the \emph{F1 property} if there is a $G^{RL} \in G^{R\mathcal{L}}$ such that $G^{RL} \geqslant G^L$ or there is a $G^{LR} \in G^{L\mathcal{R}}$ such that $G^{LR} \leqslant G^R$.
\end{definition}

\begin{definition}\label{def: F1 and F2 properties}{\rm \cite{CarvalhoHNS}}
Let $S$ be a hereditary closed ruleset. Given a position $G \in S$, the pair $(G^{L}, G^{R}) \in G^{\mathcal{L}}\times G^{\mathcal{R}}$ satisfies the \emph{F2 property} if there is a $G^{RL} \in G^{R\mathcal{L}}$ and there is a $G^{LR} \in G^{L\mathcal{R}}$ such that $G^{LR}\leqslant  G^{RL}$.
\end{definition}

As proven in~\cite{CarvalhoHNS}, if all positions of the ruleset satisfy one of these properties, the value of the position is a number.

\begin{thm}\label{thm:numbers}{\rm \cite{CarvalhoHNS}} Let $S$ be a hereditary closed ruleset. Every position in $S$ satisfies either the F1 or the F2 property if and only if every position is a number.
\end{thm}

\section{Main results}\label{sec: main results}
Before considering the values and associated strategies, we consider the outcomes; that is, we partially answer the question: ``Who wins the game?''
The full answer requires an analogous analysis to finding the values.

\begin{thm}\label{thm: left wins} Let $G=d_1d_2\ldots d_n$. If $d_1d_2\ldots d_n$ contains an even number of $1$s, or if $d_1=1$ and there are least two $1$s, then
Left wins $G$.
\end{thm}

\begin{proof} A Right move does not decrease the number of $1$s in the position. Thus, if in $G$, Left has a move, then she still has a move after any Right move in $G$. Consequently, regardless of $d_1$, if there are an even number of $1$s in $G$, it will be Left who reduces the game to all $0$s.  Similarly, if $d_1=1$ and there are an odd number of $1$s, Left will eventually reduce $G$ to a position with a single $1$; that is, to $d_1=1$ and $d_i=0$ for $i>1$. In this case, Right has no move and loses.
\end{proof}

The remaining case, $d_1=0$ and an odd number of $1$s, is more involved. The analysis of this case is the subject of the remainder of the paper. We first prove the following.

\begin{thm}\label{thm:numbers2} All \name\  positions are numbers.
\end{thm}
\begin{proof}
Let $G$ be a \name \ position.
If only one player has a move, then the game is an integer. Otherwise, let $L$ be the Left move to change $(d_i,d_j)$ from $(1,1)$ to $(0,0)$. Let $R$ be the Right move to change $(d_k,d_\ell)$ from $(0,1)$ to $(1,0)$. No other digits are changed. If all four indices are distinct, then both $L$ and $R$ can be played in either order. In this case $G^{LR}=G^{RL}$. Thus, the F2 property holds. If there are only three distinct indices, then two of the bits are ones. If Left moves first, then $d_{i} = d_{j} = d_{k} = 0$. If Right moves first, then there are still two ones remaining after his move. After Left moves, the position becomes $G(0,0,0)$ and hence, $G^{L} = G^{RL}$. The F1 property holds.

There are no more cases since there must be at least three distinct indices. Since every position satisfies either the F1 or F2 property then, by Theorem~\ref{thm:numbers}, every position is a number.
\end{proof}

Given a position $G$, the following algorithm returns a value.\\

\noindent
\textit{Algorithm:}
Let $G$ be a \name\ position. Let $G_0=G$.
\begin{enumerate}
\item Set $i=0$.
\item Reductions: Let $\alpha$ and $\beta$ be strings of coins, and either can be empty.
     \begin{enumerate}
      \item If $G_0=\alpha01^{3+j}\beta$, $j\geq 1$, then set $G_0=\alpha101^{j}\beta$.
     \item If $G_0=\alpha01^{3}\beta$, and $\beta$ contains an even number of $1$s, then set  $G_0=\alpha10\beta$.

      \item Repeat until neither case applies, then go to Step \ref{item:next}.
     \end{enumerate}
\item\label{item:next} If $G_i$ is  $0^r1$, $r\geq 0$ or  $1^a0^{p_i}10^{q_i}1$, $a\geq 0$ and $p_i+q_i\geq 0$; then  go to Step \ref{item:last}.

Otherwise,
$G_i=\alpha01^a0^{p_i}10^{q_i}1$, $p_i+q_i\geq 1$, $a>0$ and some $\alpha$.
 Set
\begin{eqnarray*}
Q_i&=&0^{p_i}10^{q_i}1,\\
G_{i+1}&=&\alpha01^a.
\end{eqnarray*}
Go to Step \ref{step:iterate}.
\item\label{step:iterate} Set $i=i+1$. Go to Step \ref{item:next}.
  \item\label{item:last} If $G_i=0^r1$, then set $v_i=-r$. If $G_i=1^a0^{p_i}10^{q_i}1$, then set $v_i=\lfloor\frac{a}{2}\rfloor + \frac{1}{2^{2p_i+q_i}}$.
  Go to Step \ref{step:os}.
  \item\label{step:os}
 For $j$ from $i-1$ down to $0$, set $v_j=v_{j+1}:\frac{1}{2^{2p_j+q_j-1}}$.
  \item Return the number $v_0$.
\end{enumerate}

The algorithm implicitly returns two different results:
\begin{enumerate}
\item  For Step 3, the substrings, $Q_0, Q_1, \ldots, Q_{i-1}, G_i$, partition the reduced version of $G$;
\item The value $v_0$.
\end{enumerate}
First we illustrate the algorithm with the following example. Let $G=10011110110110111011110011$. The reductions, applied to the underlined digits, give that:

\begin{eqnarray*}
10011110110110111011110011 &=&10011110110110111\pos{0111}10011\\
                                                    &=&1001111011011\pos{0111}1010011\\
                                                    &=&1001111011\pos{0111}01010011\\
                                                    &=&1001111\pos{0111}001010011\\
                                                    &=&10\pos{0111}110001010011\\
                                                    &=&1010110001010011.
\end{eqnarray*}
Step 3 partitions the last expression into $101(011)(000101)(0011)$ so that the ordinal sum is given by
\begin{eqnarray*}
v_0&=&\left(\left(\frac{1}{2}:\frac{1}{2}\right):\frac{1}{64}\right):\frac{1}{8}\\
  &=&\frac{10257}{16348}.
\end{eqnarray*}
Now let $H=01001110110111011101$. The reductions give that:
\begin{eqnarray*}
01001110110111011101&=&0100111011\pos{0111}011101\\
                                      &=&0100111\pos{0111}0011101\\
                                      &=&010\pos{0111}100011101\\
                                      &=&01010100011101.
\end{eqnarray*}
The last expression partitions into $01(0101)(00011)(101)$ so that
\begin{eqnarray*}
v_0&=&\left(\left(-1:\frac{1}{4}\right):\frac{1}{32}\right):1\\
  &=&-\frac{893}{1024}.\\
\end{eqnarray*}

The next theorem is the main result of the paper.

\begin{thm}[Value Theorem]\label{thm: main theorem}
Let $G$ be a \name\ position. If $v_0$ is the value obtained by the algorithm applied to $G$, then $G = v_0$.
\end{thm}
In the next section, we derive several results that will be used to prove Theorem~\ref{thm: main theorem}. The proof of Theorem~\ref{thm: main theorem} will appear in Section~\ref{sec: theorem proof}.

\section{Best moves and Reductions}\label{sec:proofs}
The proofs in this section use induction on the options. An alternate but equivalent approach is to regard the techniques as induction on the associated binary number of the positions. The proofs require detailed examination of the positions and we will use notation suitable to the case being considered. Often, a typical position will be written as a combination of generic strings
and the substring under consideration. For example, $111011000110101$ might be parsed as $(11101)(100011)(0101)$,
and written $\alpha100011\beta$ or more compactly as $\alpha10^31^2\beta$.\\

We require several results before being able to prove Theorem~\ref{thm: main theorem}. We begin by proving a simplifying reduction, followed by the best moves for each player, and then the remaining reductions used in the algorithm.

As an immediate consequence of Theorems~\ref{thm:numbers2} and~\ref{thm:literature} we have the following.

\begin{cor}\label{cor: greater than} Let $\alpha$, $\beta$, and $\gamma$ be arbitrary substrings of coins. We then have that $\alpha1\beta0\gamma > \alpha0\beta1\gamma$. Moreover, 
for an integer $r\geq 0$ we have that $\beta10^r1>\beta$.
\end{cor}
\begin{proof}
Recall that by Theorem~\ref{thm:numbers2} all \name\ positions are numbers. Thus, Theorem~\ref{thm:literature} applies.

A Right option of $\alpha0\beta1\gamma$ is $\alpha1\beta0\gamma$ and so we have that $\alpha1\beta0\gamma > \alpha0\beta1\gamma$. Similarly, a Left option of $\beta 10^r1$ is $\beta$ we have that $\beta 10^r1>\beta$. \end{proof}

Next we prove the best moves for each player. Right wants to play the zero furthest to the right and the 1 adjacent to it. Left wants to play the two ones furthest to the right.

\begin{thm}\label{thm:bestmoves}
Let $G$ be a \name\ position, where in $G$, $k$ and $n$ are the greatest indices such that the associated coin is $1$. Let $i$ be the greatest index such that $d_i=0$. Left's best move is to play $(d_{k}, d_{n})$, and Right's best move is to play $(d_{i}, d_{i+1})$.
\end{thm}

\begin{proof}
We prove this theorem by induction on the options. Note that we use the equivalent binary representation of the game position. If there are three or fewer coins, then by exhaustive analysis the theorem is true.

Let $G$ be $d_1d_2\ldots d_n$. We begin by proving Left's best moves.  Let $k$ and $n$ be the two largest indices, where $d_{k}=d_{n}=1$, and let $i$ and $j$, $i<j$, be two indices with $d_{i}=d_j=1$. We use the notation $G(d_i,d_j,d_{k},d_n)$ to highlight the salient coins. The claimed best Left move is from $G(1,1,1,1)$ to $G(1,1,0,0)$. This must be compared to any other Left move, represented by moving from  $G(1,1,1,1)$ to $G(0,0,1,1)$. That is, we need to show that
$G(1,1,0,0)-G(0,0,1,1)\geqslant 0$.
For the moves to be different, at least three of $i,j,k,n$ are distinct. Further, by the choice of the indices, we now know that $i<k$.

We first assume the four indices are distinct. In this case, we have that $i<j<k<n$.
By applying Corollary \ref{cor: greater than} twice,
we have that \[G(1,1,0,0)> G(1,0,0,1)>G(0,0,1,1).\]

We may assume then, without loss of generality, that $j=k$ or $j=n$. Now consider $G(d_i,d_{k},d_n)=G(1,1,1)$.
By Corollary \ref{cor: greater than}, we have that if $j=k$, $G(1,0,0)> G(0,0,1)$, and  if $j=n$, $G(1,0,0)> G(0,1,0)$.

We now prove Right's best move.  There are more cases to consider. Let $k$ be the largest index such that $d_k=0$ and $d_{k+1}=1$ and
also let $i,j$, $i<j$ be indices with $d_i=0$ and $d_j=1$. The claimed best move is $d_k,d_{k+1}$ and this must be compared to the arbitrary Right move
$d_i,d_j$. The original position is $$G(d_i,d_j,d_k,d_{k+1}) = G(0,1,0,1)$$ and we need to show that
$D=G(1,0,0,1)-G(0,1,1,0)\geqslant 0$. For the moves to be different, there must be at least three distinct indices.

Suppose Right plays in the first summand of $D$.
Note that, by induction, the best moves of Left and Right are known.\\

\noindent
(1) First, suppose $j<k$. By induction, Right's best move in the first summand of $D$, is to $D'=G(1,0,1,0) - G(0,1,1,0)$.
Since $i<j$, then $G(1,0,1,0)$ is a Right option of  $G(0,1,1,0)$ and thus, $D'$ is non-negative by Corollary \ref{cor: greater than}. \\

\noindent
(2)
If $j=k+1$, then there are only three distinct indices. The original game is $G(d_i,d_k,d_{k+1})=G(0,0,1)$ and $D=G(1,0,0)-G(0,1,0)$. Since  $G(1,0,0)$ is a Right option of $G(0,1,0)$, then $D$ is non-negative by Corollary \ref{cor: greater than}. \\

\noindent
(3)
If $j>k+1$, then the original game, with the indices in increasing order, is $G(d_i,d_k,d_{k+1},d_j) = G(0,0,1,1)$ and $D=G(1,0,1,0)-G(0,1,0,1)$. Right's best move is to $G(1,1,0,0)-G(0,1,0,1)$ and Left responds to $G(1,1,0,0)-G(1,1,0,0) = 0$, and Left wins. \\

In all cases, if Right moves in the first summand of $D$, then Left wins.  Next, we consider Right moving in the second summand of $D=G(1,0,0,1)-G(0,1,1,0)$. Note that by the choices of the subscripts, $d_\ell=1$ if
$n\geqslant \ell\geqslant k+1$.\\

\noindent
(1) If $n>k+2$, then Right's best move in the second summand is to change $d_{n-1},d_n$  from $(1,1)$ to $(0,0)$. Left copies this move in the first summand and the resulting difference game is non-negative by induction.\\

\noindent
(2) Suppose $n=k+2$.\\

\noindent
(2i)
If  $j< k+1$, then $G(d_i,d_j,d_k,d_{k+1},d_{k+2})=G(0,1,0,1,1)$ and  $D=G(1,0,0,1,1)-G(0,1,1,0,1)$. Right's best move is to $G(1,0,0,1,1)-G(0,1,0,0,0)$. Left moves to $G(1,0,0,0,0)-G(0,1,0,0,0)$. This is non-negative by Corollary \ref{cor: greater than} and Left wins.\\

For the next two sub-cases, exactly two $1$s will occupy two of the four indexed positions. Since Right is moving in the second summand, he is changing two $1$s to two $0$s. Thus, Left's best response for each case is to move in the first summand, bringing the game to $G(0,0,0,0)-G(0,0,0,0)=0$, and she wins. For these cases, we only list the original position. The strategy for both cases is as just described.  \\

\noindent
(2ii) If  $j= k+1$, then $G(d_i,d_k,d_{k+1},d_{k+2})=G(0,0,1,1)$ and  $D=G(1,0,0,1)-G(0,1,0,1)$. \\

\noindent
(2iii) If  $j= k+2$, then $G(d_i,d_k,d_{k+1},d_{k+2})=G(0,0,1,1)$ and  $D=G(1,0,1,0)-G(0,1,0,1)$.\\

\noindent
(3) Now suppose $n=k+1$. \\

\noindent
(3i) If  $j< k+1$, then let $ \ell<k+1$ be the largest index such that $d_\ell=1$. \\

If $j<\ell$, then we have $G(d_i,d_j,d_\ell,d_k,d_{k+1})=G(0,1,1,0,1)$ and  $D=G(1,0,1,0,1)-G(0,1,1,1,0)$.
Right's best move is to $G(1,0,1,0,1)-G(0,1,0,0,0)$. Left moves to $G(1,0,0,0,0)-G(0,1,0,0,0)$ which is non-negative since
$G(1,0,0,0,0)$ is a Right option of $G(0,1,0,0,0)$.

If $j=\ell$, then $G(d_i,d_j,d_k,d_{k+1})=G(0,1,0,1)$ and  $D=G(1,0,0,1)-G(0,1,1,0)$.
Right's best move is to $G(1,0,0,1)-G(0,0,0,0)$. Left moves to $G(0,0,0,0)-G(0,0,0,0)=0$, and Left wins.\\

\noindent
(3ii) If  $j= k+1$, then $G(d_i,d_k,d_{k+1})=G(0,0,1)$ and $D=G(1,0,0)-G(0,1,0)$. This is non-negative by Corollary \ref{cor: greater than}. \\

In all cases, Left wins $D$, proving the result.
 \end{proof}

Suppose in a position that the coins of the best Right move are different from those of the best Left move. The next lemma essentially says that the position before and after one move by each player are equal. It is phrased in a way that is useful for reducing the length of the position. For the Algorithm, it suffices to prove the result for $\beta$ being empty. However, it is useful, certainly for a human, to reduce the length of the position as much as possible.

\begin{lemm}\label{lem: reduction 1a}
If $\alpha \geq 0$, then we have that $\alpha01111^a = \alpha101^a$.
\end{lemm}
\begin{proof}
Let $H=\alpha01111^a - \alpha101^a$. We need to show $H=0$. We have several cases to consider. \\

\noindent
(1) If $a \geq 2$, then playing the same move in the other summand are good responses. After two such moves we have either
\[\alpha01111^{a-2} - \alpha101^{a-2} = 0, \quad \text{by induction,}\]
or
\[\alpha 10111^a - \alpha1101^{a-1} = \alpha1101^{a-1}-\alpha1101^{a-1} = 0, \quad \text{by induction}.\]

\noindent
(2) If $a=1$, then $H=\alpha01111-\alpha101$. The cases are:
\begin{enumerate}
\item[i.] Left plays in the first summand to $\alpha011 - \alpha101,$ then Right moves to $\alpha 101 - \alpha101=0$.
\item[ii.] Right plays in the second summand to $\alpha 01111-\alpha,$ then Left moves to $\alpha 011-\alpha$.  Since $(\alpha011)^L = \alpha$, then $\alpha011 > \alpha$.
\item[iii.] Right plays in the first summand to $\alpha10111 - \alpha101$, then Left responds to $\alpha101 - \alpha101=0$.
\item[iv.] Left plays in the second summand to  $\alpha01111-\alpha11$, then Right moves to $\alpha10111-\alpha11 = \alpha11-\alpha11 = 0$, by induction.
\end{enumerate}

\noindent
(3) If $a = 0$, then $H=\alpha0111-\alpha1$. There are several cases to consider.
\begin{enumerate}
\item[i.] If Left or Right play in the first summand, then the response is in the first summand giving $\alpha1-\alpha1=0$.
\item[ii.] If Left plays in the second summand, then since there is a Left move, then $\alpha = \beta 01^b$, $b\geq 0$. We have that $\beta01^b0111 -\beta01^b1$ and Left moves to $\beta01^b01^3 -\beta101^b$. Here, Right responds to $\beta101^{b-1}01^3 -\beta101^b$, which by induction is equal to
$ \beta101^{b-1}1 -\beta101^b =0$.
\item [iii.] Right plays in the second summand. For a Right move to exist, then $\alpha = \beta10^a$, $a\geq 0$. Thus,  
$H=\beta10^a0111-\beta10^a1$, and Right moves to  $\beta10^a0111-\beta$. Left responds by moving to $\beta00^a011-\beta$. We then have that $(\beta00^a011)^L = \beta$; thus, $\beta00^a011>\beta$. Hence, we find that $\beta00^a011-\beta>0$.
\end{enumerate}

In all cases, the second player wins $H$ thereby proving the result.
\end{proof}

There are reductions that can be applied to the middle of the position, but extra conditions are needed.
\begin{lemm}\label{lem:reduction2}
Let $\alpha$ and $\beta$ be arbitrary binary strings where either (a) $\beta$ starts with a $1$, or (b) $\beta$ starts with $0$ and has an even number of $1$s. We then have that
\[\alpha0111\beta=\alpha10\beta.\]
\end{lemm}

\begin{proof}
Let $H=\alpha0111\beta - \alpha10\beta$. We need to show that $H=0$. We have several cases to consider.

\begin{enumerate}
\item If $\beta$ is empty or $\beta =1^a$, then $H=0$ by Lemma~\ref{lem: reduction 1a}. Therefore, we may assume that $\beta$ has at least one $1$ and one $0$.

\item If $\beta = 1\gamma1$ ($\beta$ must end in a $1$),
then the best moves, in both summands, are pairs of coins in $\beta$ and $-\beta$. If each player
 copies the opponent's move in the other summand, then this leads to
\[\alpha0111\beta - \alpha10\beta \rightarrow \alpha0111\beta' - \alpha10\beta'\]
and the latter expression is equal to $0$, by induction.

\item If $\beta \neq 1\gamma1$, then $\beta = 0\gamma1$ and $\gamma1$ has at least two $1$'s. The best moves are in $\beta$ and $-\beta$ and are the best responses to each other. We then derive that
\[\alpha0111\beta-\alpha10\beta \rightarrow \alpha0111\beta'-\alpha10\beta' = 0, \quad \text{by induction.}  \]
\end{enumerate}
In all cases $H=0$, and this concludes the proof.
\end{proof}

In Lemma~\ref{lem:reduction2}, the conditions are necessary. An example is: $$3/8=011101\ne 1001=1/4.$$ Here, $\beta$ starts with a $0$ and has an odd number of $1$s.

These reduction lemmas are important in evaluating a position. The reduced positions will end in $011$ or $01$.
By considering the exact end of the string, specifically, if there are at least two 0s (in one special case three 0s), then we can find an ordinal sum decomposition.  The decomposition is determined by where the third topmost $1$ is situated.

The next result is the start of the ordinal sum decomposition of a position. The exponent is the value of the Right option of the substring being removed.
\begin{lemm}\label{lem:osvalue}
If $a \geq 1$ and $p$ and $q$ are non-negative integers such that $p+q \geq 1$,  then \[\alpha 01^a0^p10^q1=\alpha01^a:\frac{1}{2^{2p+q-1}}.\]
\end{lemm}
\begin{proof}
We prove that $$\alpha 01^a0^p10^q1-\left(\alpha01^a:\frac{1}{2^{2p+q-1}}\right)=0.$$
Note that in Theorem~\ref{thm:oso} we have that playing in the base of $\alpha 01^a: \frac{1}{2^{2p+q-1}}$ is worse than playing in the exponent. We have two cases to consider.\\

\noindent
\emph{Case 1:} Left plays in the first summand and Right in the second. Right has a move in the exponent (moves to $0$) since  $2p+q-1\geq 0$.

In either order, the final position is given by:
\[\alpha 01^a - \left(\alpha01^a:0\right) = \alpha01^a-\alpha01^a = 0.\]

\noindent
\emph{Case 2:} Right plays in the first summand and Left plays in the second summand.

 Consider

\[\alpha 01^a0^p10^q1-\left(\alpha01^a:\frac{1}{2^{2p+q-1}}\right).\]

We have two sub-cases.\\

\noindent (a) Assume $2p+q-1 \neq 0$. After the two moves we have the position
 \[\alpha01^a0^r10^s1 - \left(\alpha01^a: \frac{1}{2^{2p+q-2}}\right), \quad \text{where }2r+s = 2p+q-1.\]
By induction, we have that
\begin{align*}
\alpha01^a0^r10^s1 &= \alpha01^a:\frac{1}{2^{2r+s-1}}\\
&= \alpha01^a:\frac{1}{2^{2p+q-2}}.\\
\end{align*}

Thus, $\alpha01^a0^r10^s1 - \left(\alpha01^a:\frac{1}{2^{2p+q-2}}\right) = 0$.\\

\noindent (b) Assume $2p+q-1=0$, that is,  $q=1$, $p=0$. The original position is \[\alpha01^a101 - \left(\alpha01^a:1\right).\]

After the two moves we have the position $\alpha 01^a11-\alpha101^{a-1}$ (note that Left has no move in the exponent). By Lemma~\ref{lem: reduction 1a}, $\alpha01^a11 = \alpha101^{a-1}$. Hence, we have that $\alpha 01^a11 -\alpha101^{a-1} = 0$ and the result follows.
\end{proof}

The values of the positions not covered by Lemma \ref{lem:osvalue} are given next.
\begin{lemm}\label{lemm:value1}
Let $a$, $p$, and $q$ be non-negative integers. We then have that
\[ 0^p1=-p, \text{ \ \ and \ \  } 1^a0^p10^q1=\floor*{\frac{a}{2}}+\frac{1}{2^{2p+q}}.\]
\end{lemm}
\begin{proof}
Let $G=0^p1$. Left has no moves and Right has $p$. Note that in $1^a$,  Left has $\lfloor{\frac{a}{2}}\rfloor$ moves and Right has none.

Now, let $G=1^a0^p10^q1$.  We proceed by induction on $p+q$. In all cases, Left's move is to $1^a$, that is, to $\lfloor{\frac{a}{2}}\rfloor$. If $p=0$ and $q=0$ then $G=1^a11$, which has value $\lfloor{\frac{a}{2}}\rfloor+\frac{1}{2^{0}}= \lfloor{\frac{a}{2}}\rfloor+1$. Assume that $p+q=k$, $k>0$. If $q>0$, then $G=\{\lfloor{\frac{a}{2}}\rfloor \mid 1^a0^p10^{q-1}1\}$. By induction, we have that \[G=\left\{\floor*{\frac{a}{2}}\,\Big\vert \,\floor*{\frac{a}{2}}+\frac{1}{2^{2p+q-1}}\right\} = \floor*{\frac{a}{2}}+\frac{1}{2^{2p+q}}.\]

If $q=0$, then $G=\{\lfloor{\frac{a}{2}}\rfloor \mid 1^a0^{p-1}10^{1}1\}$. By induction, we have that \[G=\left\{\floor*{\frac{a}{2}} \,\Big\vert \, \floor*{\frac{a}{2}}+\frac{1}{2^{2(p-1)+1}}\right\} = \floor*{\frac{a}{2}}+\frac{1}{2^{2p}},\]
and the result follows.
\end{proof}

\subsection{Proof of the Value Theorem}\label{sec: theorem proof}
We now have all of the tools to prove Theorem~\ref{thm: main theorem}.\\

\noindent
\begin{proof}[Proof of Theorem~\ref{thm: main theorem}] Let $G$ be a \name\ position. Step 2 reduces the string of coins. 
 The reductions in Step 2(a) are those of Lemma~\ref{lem: reduction 1a} and Lemma~\ref{lem:reduction2} part(a). The reductions in Step 2(b) are those of Lemma~\ref{lem:reduction2} part(b). In all cases, these lemmas show that each new reduced position is equal to $G$.

In Step 3, we claim $G_i\ne\beta1^3$ for any $\beta$. This is true for $i=0$ by Lemma~ \ref{lem: reduction 1a}. If $i>0$, then at each iteration of Step 3, the last two $1$s are removed from $G_{i-1}$. Now,
the original reduced position
 would be $G_0=\beta1^3\gamma$, where $\gamma$ has an even number of $1$s. Lemma~\ref{lem:reduction2} part(b) would apply eliminating the three consecutive $1$s.
Now either $G_i$ is one of $0^r1$, $r\geq 0$ or  $1^a0^{p_i}10^{q_i}1$, $a\geq 0$ and $p_i+q_i\geq 0$, or
$G_i=\alpha01^a0^{p_i}10^{q_i}1$, $p_i+q_i\geq 1$, $a>0$. In the latter case, the index is incremented and the algorithm goes back to Step 3.

 Step 5 applies when Step 3 no longer applies, i.e., $G_i$ is one of $0^r1$, $r\geq 0$ or  $1^a0^{p_i}10^{q_i}1$, $a\geq 0$ and $p_i+q_i\geq 0$. Now, $v_i$ is the value of $G_i$, as given in Lemma  \ref{lemm:value1}.

 Lemma~\ref{lem:osvalue} shows that for each $j<i$, $G_{j} = G_{j+1}:\frac{1}{2^{2p_j+q_j}}$, the evaluation in Step 6.  Thus, the value of $G$ is $v_0$, and the theorem follows. \end{proof}

The question: ``Who wins $0101011111+1101100111+0110110110111$ and how?'' from Section~\ref{secintro} can now be answered.

 First, we have that
 \begin{eqnarray*}0101011111 &=& 01011011 = \left(01011:\frac{1}{2}\right) = \left(\left(01:\frac{1}{2}\right):\frac{1}{2}\right) \\
 &=& \left(\left(-1:\frac{1}{2}\right):\frac{1}{2}\right)=-\frac{11}{16}\\
 1101100111&=&1101101 = (1101:1) = \left(\frac{1}{2}:1 \right) = \frac{3}{4}\\
0110110110111&=&0110110111=0110111 =0111= 0.\\
 \end{eqnarray*}

Thus, we have that \[0101011111+1101100111+0110110110111=-\frac{11}{16}+\frac{3}{4}+0=\frac{1}{16}.\] Left's only winning move is to \[01010111+1101100111+0110110110111=-\frac{3}{4}+\frac{3}{4}+0=0.\] Her best moves in the second position gives a sum of $-\frac{11}{16}+\frac{5}{8}+0 = -\frac{1}{16}$, and in the third yields
$-\frac{11}{16}+\frac{3}{4}-\frac{1}{8}=-\frac{1}{16}$. Left loses both times.

 \section{Open Problems}

 Natural variants of \textsc{flipping coins} involve increasing the number of coins that can be flipped from two to three or more. A brief computer search suggests that the only version where the values are numbers is the game in which Left flips a subsequence of all $1$s and Right a subsequence of $0$s ended by a $1$. We conjecture that a similar ordinal sum structure will arise in these variants. Other variants have values that include \textit{switches, tinies, minies,} and other three-stop games.  However, some variants, when the reduced canonical values are considered,  only seem to consist of numbers and switches. A more thorough investigation should shed light on their structures.

Instead of playing on a string, we can play on a directed acyclic graph (or DAG). We know that the impartial game \textsc{twins}, played on a DAG, reduces to a simple multi-heap game. Of course, instead of coins, other objects could be used. In a line of dice, for example, the end dice would be turned to a lower number.

\end{document}